\documentclass[11pt]{amsart}
\usepackage{amscd}
\def\NZQ{\Bbb}               % the font for N,Z,Q,R,C
\def\NN{{\NZQ N}}

\def\ZZ{{\NZQ Z}}

\def\frk{\frak}               % font for "Fraktur"

\def\Phi{{\frk n}}
\def\Phi{{\frk N}}
\def\opn#1#2{\def#1{\operatorname{#2}}} % to make operators
\opn\chara{char} \opn\length{\ell} \opn\pd{pd} \opn\rk{rk}
\opn\projdim{proj\,dim} \opn\injdim{inj\,dim} \opn\rank{rank}
\opn\depth{depth} \opn\grade{grade} \opn\height{height}
\opn\embdim{emb\,dim} \opn\codim{codim}

\opn\Tr{Tr} \opn\bigrank{big\,rank}
\opn\superheight{superheight}\opn\lcm{lcm}
\opn\trdeg{tr\,deg}%\emph{
\opn\reg{reg} \opn\lreg{lreg} \opn\ini{in} \opn\lpd{lpd}
\opn\size{size}
\opn\div{div} \opn\Div{Div} \opn\cl{cl} \opn\Cl{Cl}
\opn\Spec{Spec} \opn\Supp{Supp} \opn\supp{supp} \opn\Sing{Sing}
\opn\Ass{Ass} \opn\Min{Min}
\opn\Ann{Ann} \opn\Rad{Rad} \opn\Soc{Soc}
\opn\Im{Im} \opn\Ker{Ker} \opn\Coker{Coker} \opn\Am{Am}
\opn\Hom{Hom} \opn\Tor{Tor} \opn\Ext{Ext} \opn\End{End}
\opn\Aut{Aut} \opn\id{id}

\opn\nat{nat}
\opn\pff{pf}%   \pf exists already
\opn\Pf{Pf} \opn\GL{GL} \opn\SL{SL} \opn\mod{mod} \opn\ord{ord}
\opn\Gin{Gin} \opn\Hilb{Hilb}\opn\sdepth{sdepth}
\opn\aff{aff} \opn\con{conv} \opn\relint{relint} \opn\st{st}
\opn\lk{lk} \opn\cn{cn} \opn\core{core} \opn\vol{vol}
\opn\link{link} \opn\star{star}
\opn\gr{gr}

\def\pot#1#2{#1[\kern-0.28ex[#2]\kern-0.28ex]}

\opn\dirlim{\underrightarrow{\lim}}
\opn\inivlim{\underleftarrow{\lim}}
\let\Dirsum=\bigoplus

\def\Implies{\ifmmode\Longrightarrow \else
        \unskip${}\Longrightarrow{}$\ignorespaces\fi}
\def\implies{\ifmmode\Rightarrow \else
        \unskip${}\Rightarrow{}$\ignorespaces\fi}
\def\iff{\ifmmode\Longleftrightarrow \else
        \unskip${}\Longleftrightarrow{}$\ignorespaces\fi}

\let\:=\colon
\newtheorem{Theorem}{Theorem}[section]
\newtheorem{Lemma}[Theorem]{Lemma}
\newtheorem{Corollary}[Theorem]{Corollary}
\newtheorem{Proposition}[Theorem]{Proposition}

\newtheorem{Example}[Theorem]{Example}

\let\epsilon\varepsilon
\let\phi=\varphi
\let\kappa=\varkappa
\def\qed{\ifhmode\textqed\fi
      \ifmmode\ifinner\quad\qedsymbol\else\dispqed\fi\fi}
\def\textqed{\unskip\nobreak\penalty50
       \hskip2em\hbox{}\nobreak\hfil\qedsymbol
       \parfillskip=0pt \finalhyphendemerits=0}
\def\dispqed{\rlap{\qquad\qedsymbol}}

\opn\dis{dis}
\def\pnt{{\raise0.5mm\hbox{\large\bf.}}}

\opn\Lex{Lex}

\begin{document}

\title{Depth and Stanley depth of multigraded modules}

\author{Asia Rauf}
\thanks{The author wants to acknowledge Higher Education Commission Pakistan for partial financial support during preparation of this work.}
\address{Asia Rauf,
Abdus Salam School of Mathematical Sciences, GC University, Lahore}
\email{asia.rauf@gmail.com}

\begin{abstract}
We study the behavior of depth and Stanley depth along short exact
sequences of multigraded modules and under reduction modulo an element.
\end{abstract}

\maketitle

\section*{Introduction}
Let $K$ be a field and $S=K[x_1,\ldots,x_n]$  a polynomial ring in
$n$ variables over $K$. Let $M$ be a finitely generated
multigraded (i.e. $\ZZ^n$-graded) $S$-module. Let $m\in M$ be a
homogeneous element in $M$ and $Z\subseteq \{x_1,\ldots,x_n\}$. We
denote by $mK[Z]$ the $K$-subspace of $M$ generated by all elements
$mv$, where $v$ is a monomial in $K[Z]$. The multigraded $K$-subspace
$mK[Z]\subset M$ is called Stanley space of dimension $|Z|$, if
$mK[Z]$ is a free $K[Z]$-module. A Stanley decomposition of $M$ is a
presentation of the  $K$-vector space $M$ as a finite direct sum of
Stanley spaces $\mathcal{D}:\,\,M=\Dirsum_{i=1}^rm_iK[Z_i]$. Set
$\sdepth \mathcal{D}=\min\{|Z_i|:i=1,\ldots,r\}$. The number
\[
\sdepth(M):=\max\{\sdepth({\mathcal D}):\; {\mathcal D}\; \text{is a
Stanley decomposition of}\;  M \}
\]
is called Stanley depth of $M$. In $1982$, Richard P. Stanley
\cite[Conjecture 5.1]{St1} conjectured that $\sdepth(M)\geq
\depth(M)$
  for all finitely generated $\ZZ^n$-graded $S$-modules $M$.
    The conjecture is discussed in some special cases in
    \cite{HSY}, \cite{I2}, \cite{HVZ}, \cite{As}, \cite{I1}, \cite{P}
    \cite{C1}, \cite{C2}.

If $M$ is a finitely generated module over a Noetherian local ring $(R,m)$ and
$x\in m$ then it is well-known that $\dim M/xM\geq \dim M-1$.
   Our Proposition \ref{prodepth} and Lemma \ref{old}, show that the above inequality is preserved for $\depth$
   and $\sdepth$ when $M=S/I$ and $I\subset S$ is a monomial ideal and $x=x_k$ for any $k\in[n]$.
   If $M$ is a general
    multigraded $S$-module, then we might have $\depth M/x_kM < \depth M-1$ as shows  Example
    \ref{d1}. Also we might have $\sdepth M/x_kM < \sdepth M-1$ even if $x_k$ is regular on $M$, as shows Example \ref{me}.

As we know depth decreases by one if we reduce modulo a regular
element. In \cite[Theorem 1.1]{As}, it is proved that the
corresponding statement holds for the Stanley depth in the case of
$M=S/I$ where $I\subset S$ is a monomial ideal and $f$  a monomial
in $S$ which is regular on $M$. The next question arises whether
this is true for any multigraded module? The answer is no, see
Example \ref{me}. Let
$$\mathcal{F}:\,\,0=M_0\subset M_1\subset \ldots \subset M_r=M$$ be a chain of $\ZZ^n$-graded submodules of $M$. Then $\mathcal{F}$ is called a prime filtration of $M$ if $M_i/M_{i-1}\cong (S/P_i)(-a_i)$ where $a_i\in\ZZ^n$ and $P_i$ is a monomial prime ideal for all $i$. We denote $\Supp \mathcal{F}=\{P_1,\ldots,P_r\}$. A finitely generated module $M$ is called almost clean if there exists a prime filtration $\mathcal{F}$ of $M$ such that $\Supp(\mathcal{F})=\Ass(M)$. We show in Lemma \ref{aclean} that for almost clean module $M$ and $x_k\in S$ being regular on $M$, we have $\sdepth M/x_kM \geq\sdepth M - 1.$
However, we show in Proposition \ref{red} that
$\sdepth M/x_kM\leq \sdepth M - 1$, if $x_k$ is regular on $M$. As
an application we get that Stanley's conjecture holds for $M$ if it
holds for the module $M/x_kM$ (see Corollary \ref{imp}). Moreover if
$M$ has a maximal regular sequence given by monomials then Stanley's conjecture holds for $M$ (see Corollary \ref{main}).

Given a short exact sequence of finitely generated multigraded
$S$-modules, then the Stanley depth of the middle one is greater than or equal to the minimum of Stanley depths of the ends (see Lemma
\ref{exact}). Several examples show that the "Depth lemma" is mainly
wrong in the frame of $\sdepth$ (see Examples \ref{me2} and
\ref{d2}). However, we prove in Lemma \ref{mcii} that if $I$ is any
monomial complete intersection of $S$, then $\sdepth I$ is greater than or equal to $\sdepth S/I+1$. But in general for any monomial ideal this inequality is still an open question.

In the last section, we prove that if $I\subset
S_1=K[x_1,\ldots,x_n]$, $J\subset
 S_2=K[y_1,\ldots,y_m]$ are monomial ideals and
 $S=K[x_1,\ldots,x_n,y_1,\ldots,y_m]$, then the Stanley depth of the tensor product of $S_1/I$
 and $S_2/J$ (over $K$) is greater than or equal to the sum of $\sdepth S_1/I$ and $\sdepth S_2/J$. This inequality could be strict as shows Example \ref{atpexp}.

I am grateful to Professor J\"urgen Herzog for useful discussions and comments during the preparation of the paper.
%I also want to thank referee for Corollary \ref{ref1}, Example \ref{ref2} and Example \ref{ref3}.

\section{The behavior of depth and sdepth under reduction modulo element}

In dimension theory it is well known the following result (see e.g.
\cite[Proposition A.4]{BH}, or \cite[Corollary 10.9]{Eis})

\begin{Theorem}
If $(R,m)$ is a Noetherian local ring and $M$ is finitely generated R-module, then for any $x\in m$ we have
\[
\dim M/xM\geq \dim M-1.
\]
\end{Theorem}

If we consider reduction by a regular element, then the $\depth$
decreases by one. But what happens if we take reduction
 by a non-regular element?

\begin{Proposition}
\label{prodepth}Let $S=K[x_1,\ldots,x_{n}]$ be a polynomial ring
over a field $K$, $I\subset S$  a monomial ideal and  $R=S/I$. Then
\begin{eqnarray}
\label{depth1}\depth(R/x_nR)\geq \depth(R)-1.
\end{eqnarray}
\end{Proposition}

\begin{proof}
Let $\bar{R}=R/x_nR\simeq S/(I,x_n)$. We denote
$\bar{S}=K[x_1,\ldots,x_{n-1}]$ and let $x'=\{x_1,\ldots,x_{n-1}\}$,
$x=\{x_1,\ldots,x_n\}$.

Let $\varphi$ be the canonical map from $R$ to $\bar{R}$ and
$\alpha$ be the composite map
$$\bar{S}\longrightarrow S \longrightarrow R=S/I,$$ where the first map is
the canonical embedding and
the second map is the canonical surjection. It
is clear that $\ker(\alpha)=I\cap \bar{S}. $  Let $\alpha_1$ be the composite map
$$\bar{S}\longrightarrow S \longrightarrow R=S/I \longrightarrow S/(I,x_n).$$ It is clear that
$\alpha_1$ is surjective. We claim that $\ker(\alpha_1)=I\cap \bar{S}$. One inclusion is obvious.
 To prove other inclusion, we consider a monomial $v\in \ker(\alpha_1)$, that is, $v\in(I,x_n)$.
 Since $v\in\bar{S}$ and $I$ is a monomial ideal, it follows that $v\in I$.  Let $\ker(\alpha_1)=\bar{I}$ then
$\bar{S}/\bar{I}\simeq S/(I,x_n)$. It follows that the composition
$\bar{R}\rightarrow R\rightarrow \bar{R}$ of the natural maps is the
identity. Therefore, the $S$-module $\bar{R}$ is a direct summand of
the $S$-module $R$. This implies that the $S$-module
$H_i(x';\bar{R})$ is a direct summand of $H_i(x';R)$ for all $i$,
where $H_i(x';\bar{R})$ and $H_i(x';R)$ are the i-th Koszul homology
modules of $x'$ with respect to $\bar{R}$ and $R$ respectively. In
particular, if $H_i(x';\bar{R})\neq 0$, then $H_i(x';R)\neq0$. Let
$k=\max\{i\mid H_i(x':\bar{R})\neq0\}$. Then $\depth\bar{R}=n-1-k$,
by \cite[Theorem 1.6.17]{BH}. Since $H_k(x';\bar{R})\neq0$, it
follows that $H_k(x';R)\neq0$ which implies that $H_k(x;R)\neq0$ by
\cite[Lemma 1.6.18]{BH}. Therefore applying again
\cite[Theorem 1.6.17]{BH} it follows that $\depth R\leq
n-k=\depth\bar{R}+1$.
\end{proof}

\begin{Corollary}
\label{ref1}Let $I\subset S$ be a monomial ideal. Then $\depth S/(I:u)\geq \depth S/I$ for all monomials $u\not\in I$.
\end{Corollary}
\begin{proof}
Since $(I:uv)=((I:u):v)$, where $I$ is a monomial ideal and $u$ and $v$ are monomials, we may reduce to the case $u=x_n$, and apply recurrence. Then we have the exact sequence
\[
0\longrightarrow S/(I:x_n)\longrightarrow S/I\longrightarrow S/(I,x_n)\longrightarrow 0.
\]
By Depth Lemma \cite[Lemma 1.3.9]{V} and Proposition \ref{prodepth}, we obtain the required result.
\end{proof}

This Corollary does not hold (and so Proposition \ref{prodepth}) if $u$ is not a monomial, as we have the following example:

\begin{Example}
\label{ref2}{\em Let $S=K[x,y,z,t]$ and $I=(x,y)\cap (y,z)\cap (z,t)$ and $u=y+z$. Then $J:=(I:u)=(x,y)\cap (z,t)$ and $\depth S/J=1<2=\depth S/I$.}
\end{Example}

The Proposition \ref{prodepth} is not true in general. If $M$ is a finitely
generated graded $R$-module and $x\in R_1$ then we might have
\[
\depth(M/xM)< \depth(M)-1,
\]
as shows the following example:
\begin{Example}{\em
\label{d1}Let $S=K[x,y,z,t]$, $M=(x,y,z)/(xt)$. We have $\depth M=2$
and $M/xM=(x,y,z)/(x^2,xy,xz,xt)$. Since the maximal ideal is an
associated prime ideal of $M/xM$, we get $\depth M/xM=0$. Hence
$\depth(M/xM)< \depth(M)-1$.}
\end{Example}

In Proposition \ref{prodepth} we might have
\[
\depth(R/x_nR)>\depth(R)-1,
\]
as shows the following example:

\begin{Example}{\em
 Let $I=(x_1^2,x_1x_2,\ldots,x_1x_n)\subset S=K[x_1,\ldots,x_{n}]$ be a monomial ideal of $S$ and
 $R=S/I$.
  Then  $\depth(R)=0$ since the maximal ideal $(x_1,x_2,\ldots,x_n)\in \Ass(R)$.  Since
  $R/x_1R=S/(x_1)\simeq K[x_2,\ldots,x_n]$, we get $\depth(R/x_1R)=n-1$. Hence
  $\depth(R/x_1R)>  \depth(R)-1$.}
\end{Example}

For the $\sdepth$ we have a statement similar to that of Proposition \ref{prodepth}. Indeed in the proof of \cite[Lemma 1.2]{As} where it was shown that $\sdepth S/(I,x_n)=\sdepth S/I - 1$ if $x_n$ is regular on $S/I$, we actually showed the following (without any assumption on $x_n$):
\begin{Lemma}
\label{old}Let $S=K[x_1,\ldots,x_{n}]$ be a polynomial ring over the field $K$.
Let $I\subset S$ be any monomial ideal. Then
\[
\sdepth(S/(I,x_n))\geq \sdepth(S/I)-1.
\]
\end{Lemma}
This lemma can not be extended to general multigraded modules $M$ as shows the following example, where the variable is even regular on $M$.
\begin{Example}
\label{me}{\em Let $M=(x,y,z)$ be an ideal of $S=K[x,y,z]$. Consider
a Stanley decomposition $M=zK[x,z]\oplus xK[x,y]\oplus yK[y,z]\oplus
xyzK[x,y,z]$. Since $\sdepth M\leq \dim S=3$
 and  $M$ is not a principle ideal, it follows $\sdepth M=2$.
Note that $x$ induces a non-zero element  in the socle of $M/xM$
which cannot be contained in any Stanley space of dimension greater
or equal with one. Hence
   $\sdepth M/xM=0$. Thus $\sdepth M/xM< \sdepth M - 1$.}
\end{Example}

However for the special case when $M$ is {\em almost clean} (see
\cite{HVZ}), that is there exists a prime filtration $\mathcal{F}$
of $M$ such that
  $\Supp(\mathcal{F})=\Ass(M)$, we have the following:

  \begin{Lemma} \label{aclean} Let $M$ be a finitely generated multigraded $S$-module. If $M$
is almost clean and $x_k\in S$ is regular on $M$, then
\[
\sdepth M/x_kM \geq\sdepth M - 1.
\]
\end{Lemma}
\begin{proof}
Suppose that  $\mathcal{F}$ is given by
$$0=M_0\subset M_1\subset \ldots \subset M_r=M$$
with $M_i/M_{i-1}\cong S/P_i(-a_i)$ for some $a_i\in \NN^n$ and some
monomial prime ideals $P_i$. Since $\Ass M=\{P_1,\ldots,P_r\}$ and
$x_k$ is regular on $M$, we get $x_k\not \in P_i$ and so $x_k$ is
regular on $M_i/M_{i-1}$. Set ${\bar M}_i=M_i/x_kM_i$. Then ${\bar
M}_i\subset {\bar M}_{i+1}$ and $\{{\bar M}_i\}$ define a filtration
${\bar {\mathcal{F}}}$ of ${\bar M}={\bar M}_r$ with ${\bar
M}_i/{\bar M}_{i-1}\cong S/(P_i,x_k)(-a_i)$. Thus $\sdepth {\bar
M}\geq \min_i \dim S/(P_i,x_k)=\sdepth M-1$ (see Corollary
\ref{easy}).
\end{proof}

The above Example \ref{me} hints that if $x$ is a regular element on $M$,
then $\sdepth M/xM\leq \sdepth M - 1$. This is the subject of our
next proposition.
\begin{Proposition}
\label{red}Let $M$ be finitely generated $\ZZ^n$-graded $S$-module and let $x_k$ be regular on $M$.
If $\mathcal{D}_1:\;M/x_kM=\Dirsum_{i=1}^r\bar{m_i}K[Z_i]$, is a Stanley decomposition of $M/x_kM$,
where $m_i\in M$ is homogeneous and $\bar{m_i}=m_i+x_kM$. Then
\begin{eqnarray}
M=\Dirsum_{i=1}^rm_iK[Z_i,x_k]
\end{eqnarray}
is a Stanley decomposition of $M$. In particular
\[
 \sdepth M/x_kM\leq \sdepth M - 1.
\]
\end{Proposition}
\begin{proof}

Let $N=\sum_{i=1}^rm_iK[Z_i,x_k]$. Then $N\subseteq M$. Since $\mathcal{D}_1$ is a Stanley decomposition of $M/x_kM$ it follows that $\psi(N)=M/x_kM$ where $\psi:\,M\rightarrow M/x_kM$ is the canonical epimorphism. This implies that $M=x_kM+N$ as $\ZZ^n$-graded $K$ vector spaces. We show that $M=N$. First we observe that $M=x_k^dM + N$ for all $d$. This follows by induction on $d$, because if we have $M=x_k^{d-1}M+N$, then
$M=x_k^{d-1}(x_kM+N)+N=x_k^{d}M+x_k^{d-1}N+N=x_k^{d}M+N$ since
$x_k^{d-1}N\subset N$. This completes the induction. Since $M$ is finitely generated there exists an integer $c$ such that $\deg_{x_k}(m)\geq c$ for all homogeneous elements $m\in M$.  Now let $m\in M$ be a homogeneous element with $\deg_{x_k}(m)=a$ and let $d>a-c$ be an integer. Since $M=x_k ^dM+N$, there exist homogeneous elements $v\in M$ and $w\in N$ such that $m=x_k^dv+w$, where $a=\deg_{x_k} v+d=\deg_{x_k}w$. It follows that $\deg_{x_k} v=a-d<c$, a contradiction. It implies that $v=0$, hence $m=w\in N$.

Now we show that the sum $\sum_{i=1}^rm_iK[Z_i,x_k]$ is direct, that
is $$m_iK[Z_i,x_k]\cap\sum_{j=1 \atop j\neq i}^rm_jK[Z_j,x_k]=(0).$$
Let $u=m_iq_i=\sum_{j=1 \atop j\neq i} ^rm_jq_j\in M$ be homogeneous for some $q_j$
monomials in $K[Z_j,x_k]$ such that $\deg(u)=\deg(m_jq_j)$ for all $j$. Let $p$ be the biggest power of $x_k$
dividing $q_i$. If $p=0$, then we have $\bar{u}=\bar{m_i}q_i\neq 0$ in $M/x_kM$ since $\bar{m_i}K[Z_i]$ is a Stanley space. It follows that $\bar{u}\in \bar{m_i}K[Z_i]\cap\sum_{j=1 \atop j\neq i} ^r\bar{m_j}K[Z_j]$, a contradiction. In the case of $p>0$, then in $M/x_kM$ we get
$\bar{u}=0=\sum_{j=1 \atop j\neq i} ^r\bar{m_j}\bar{q_j}$. It
follows that $\bar{q_j}=0$, since $\mathcal{D}_1$ is a Stanley decomposition of $M/x_kM$. Thus
$q_j=x_kq_j'$ for some $q_j'\in K[Z_j,x_k]$ and we get
$x_k(m_iq_i'-\sum_{j=1 \atop j\neq i} ^rm_jq_j')=0$, which implies
$m_iq_i'-\sum_{j=1 \atop j\neq i} ^rm_jq_j'=0$ since $x_k$ is
regular on $M$. Applying the same argument by recurrence we get
$q_j=x_k^ps_{j}$ for some $s_{j}\in K[Z_j,x_k]$, and $m_is_{i}=\sum_{j=1 \atop
j\neq i} ^rm_js_{j}.$ We set $v=m_is_{i}.$ Since $\bar{s_i}\neq 0$, we get $\bar{v}\neq 0$ because $\bar{m_i}K[Z_i]$ is a Stanley space. On the other hand $\bar{v}\in
\bar{m_i}K[Z_i]\cap\sum_{j=1 \atop j\neq i} ^r\bar{m_j}K[Z_j]$. It implies that $\bar{v}=0$, a contradiction.

Finally we show that each $m_iK[Z_i,x_k]$ is a
Stanley space. Indeed, suppose that $m_if=0$ for some $f\in K[Z_i,x_k]$ where $f=\sum_{j=0} ^af_jx_k^j$ such that $x_k$ does not divide $f_j$ for all $j$ then $\sum_{j=0} ^am_if_jx_k^j=0$ implies that $\bar{m_i}f_0=0$ in $M/x_kM$. We get $f_0=0$ since $\bar{m_i}K[Z_i]$ is a Stanley space. It follows that $f=x_kg$ where $g=\sum_{j=1} ^af_jx_k^{j-1}$ and from $x_km_ig=m_if=0$ we get $m_ig=0$, $x_k$ being regular on $M$. Then induction on the degree of $f$ concludes the proof since $\deg_{x_k}g<\deg_{x_k}f$.
\end{proof}
\begin{Corollary}
\label{imp}If Stanley's conjecture holds for the module $M/x_iM$, where $x_i\in S$ is regular on $M$, then it also holds for $M$.
\end{Corollary}
\begin{Corollary} \label{old1}
The equality holds in Lemma \ref{old}, if $x_n$ is regular on $S/I$.
\end{Corollary}
The proof follows from Lemma \ref{old} and Proposition \ref{red} for
$M=S/I$.

\begin{Corollary}
\label{main} Let $\depth M=t$. If there exists $u=u_1,\ldots,u_t\in
Mon(S)$ such that $u$ is regular sequence on $M$ then Stanley's
conjecture holds for $M$.
\end{Corollary}

\begin{proof}
For any regular sequence $u=u_1,\ldots,u_t\in Mon(S)$ we may choose $u$ such that $u_i=x_{i_j}$ for all $1\leq i\leq t$,  where $x_{i_j}\in \supp(u_i)$, since for any monomial $u_i\in S$ being regular on $M$ implies that each $x_{i_j}\in \supp(u_i)$ is regular on $M$, because if $x_{i_j}$ belong to the set of zero divisors of $M$ then $x_{i_j}\in P$ for some $P\in \Ass(M)$, so $u_i\in P$, which is not true as $u_i$ is regular on $M$. Since $u$ is a maximal regular sequence on $M$, we have $\depth
M/(u_1,\ldots,u_t)M=0$. Applying Proposition \ref{red} by recurrence
we get $\sdepth M\geq \sdepth M/(u_1,\ldots,u_t)M+t \geq t=\depth M$.
Hence Stanley's conjecture holds for  $M$.
\end{proof}
\begin{Example}{\em
Let $S=K[x,y,z,t]$ and $M=(x,y,z)/(xy)$. Since $\depth M=2$ and
$\{z,t\}$ is a $M$-regular sequence,  we may apply Corollary
\ref{main} to see that Stanley's conjecture holds for $M$.}
\end{Example}

\begin{Theorem}
\label{acleanth}Let $M$ be a finitely  generated multigraded
$S$-module. If $M$ is almost clean and $x_k\in S$ is regular on $M$,
then
\[
\sdepth M/x_kM =\sdepth M - 1.
\]
\end{Theorem}
The proof follows from Lemma \ref{aclean} and Proposition \ref{red}.
\begin{Theorem}
\label{aclean3}Let $M$ be a finitely generated multigraded
$S$-module. If $M$ is almost clean and $u\in S$ is a monomial, which
is regular on $M$, then $ \sdepth M/uM \geq \sdepth M - 1.$
%In particular, Stanley conjecture holds for $M$ if and only if it holds for $M/uM$.
\end{Theorem}
\begin{proof}
Let $u=x_{i_1} ^{a_1}\ldots x_{i_t} ^{a_{t}}$. Since $u$ is regular
on $M$, it follows that each $x_{i_k}\in \supp(u)$ is regular on
$M$, where we denote by $\supp(u)$ the set of all variables $x_j$
such that $x_j$ divides the monomial $u$. We consider an ascending
chain of submodules of $M$ between $uM$ and $M$ where two successive
members of the chain are of the form
\[
x_{i_1} ^{b_1}\cdots x_{i_k} ^{b_k} \cdots x_{i_t} ^{b_{t}}M\subset
x_{i_1} ^{b_1}\cdots x_{i_k} ^{b_k-1} \cdots x_{i_t} ^{b_{t}}M,
\]
and where $b_i\leq a_i$ for all $i=1,\ldots, t$.

We obtain
\[
x_{i_1} ^{b_1}\cdots x_{i_k} ^{b_k-1} \cdots x_{i_t}
^{b_{t}}M/x_{i_1} ^{b_1}\cdots x_{i_k} ^{b_k} \cdots x_{i_t}
^{b_{t}}M
 \simeq
M/x_{i_k}M,
\]
since each $x_{i_k}\in\supp(u)$ is regular on $M$. Therefore Lemma
\ref{aclean} and Corollary \ref{easy} imply that
\[
\sdepth(M/uM)\geq  \sdepth(M/x_{i_k}M) = \sdepth M-1.
\]
\end{proof}

\section{The behavior of sdepth on short exact sequence of multigraded modules}

The following "Depth Lemma" is well-known.
\begin{Lemma}\cite[Lemma 1.3.9]{V}
\label{depthlemma}If
\[
0\rightarrow U\rightarrow M\rightarrow N\rightarrow 0
\]
is a short exact sequence of modules over a local ring $R$, then
\begin{enumerate}
  \item[(a)] If $\depth M<\depth N$, then $\depth U=\depth M$.
%  \item[(b)] If $\depth M=\depth N$, then $\depth U\geq \depth M$.
  \item[(b)] If $\depth M>\depth N$, then $\depth U=\depth N + 1$.
\end{enumerate}
\end{Lemma}
We will show that most of the statements of the "Depth Lemma" are wrong if we replace $\depth$ by $\sdepth$. We first observe

\begin{Lemma}
\label{exact}Let
\[
0\rightarrow U\xrightarrow{f} M\xrightarrow{g} N\rightarrow 0
\]
be an exact sequence of finitely generated $\ZZ^n$-graded
$S$-modules. Then \[ \sdepth M\geq \min\{\sdepth U,\sdepth N\}
\]
\end{Lemma}

\begin{proof}
Let $\mathcal D:\,\,U=\bigoplus_{i=1} ^ru_iK[Z_i]$ be a Stanley
decomposition of $U$ with $\sdepth(\mathcal D)=\sdepth U$ and let
$\mathcal {D}':\,\,N=\bigoplus_{j=1} ^sn_jK[Z_j ']$ be a Stanley
decomposition of $N$ with $\sdepth(\mathcal D')\\ =\sdepth N$. Since $f$ is injective map, we may suppose that $f$ is an inclusion.
Let $n_j '\in M$ be a $\ZZ^n$ homogeneous element such that $g(n_j
')=n_j$.
Clearly, $M=\sum_{i=1}^ru_iK[Z_i]+\sum_{j=1} ^sn_j 'K[Z_j '].$ We prove that the sum $\sum_{i=1}^ru_iK[Z_i]+\sum_{j=1} ^sn_j
'K[Z_j ']$ is direct. Set $V=\sum_j n_j 'K[Z_j']$. Since the exact sequence splits as linear spaces we see that $U\cap V=\{0\}$. Clearly $\mathcal D$ is already a Stanley decomposition of $U$ and remains to show only that if $y\in n_j 'K[Z_j']\cap \sum_{k=1\atop k\neq j}^sn_k 'K[Z_k ']$ then $y=0$. As $g(y)\in n_jK[Z_j']\cap \sum_{k=1\atop k\neq j}^sn_kK[Z_k ']=\{0\}$, we see that $y\in U$, that is $y\in U\cap V=\{0\}$.
\end{proof}

\begin{Corollary}
\label{easy}Let \[ (0)=M_0\subset M_1\subset \ldots \subset
M_{r-1}\subset M_r=M
\]
be an ascending chain of $\ZZ^n$-graded submodules of $M$. Then
\begin{eqnarray}
\label{easyineq} \sdepth M\geq \min\{\sdepth M_i/M_{i-1}:\; i\in
\{1,\ldots,r\}\}
\end{eqnarray}
for all $i\in[r]$.
\end{Corollary}
\begin{proof}
We consider the exact sequence of $\ZZ^n$-graded submodules of $M$
such that
\[
0\rightarrow M_{i-1}\rightarrow M_{i}\rightarrow
M_{i}/M_{i-1}\rightarrow 0.
\]
By Lemma \ref{exact}, we get $\sdepth M_i\geq \min\{\sdepth M_{i-1},
\sdepth M_i/M_{i-1}\}.$ We apply induction to prove the inequality
(\ref{easyineq}). For $i=1$ this holds clearly. We suppose
(\ref{easyineq}) is true for $i=t$ then we have \[ \sdepth M_t\geq
\min\{\sdepth M_i/M_{i-1}:\; i\in \{1,\ldots,t\}\}.
\] Let $i=t+1$ then we have $\sdepth M_{t+1}\geq \min\{\sdepth M_{t},
\sdepth M_{t+1}/M_{t}\}$, which is enough.
\end{proof}

The analogue of Lemma $\ref{depthlemma}(a)$ only holds under an additional assumption.

\begin{Corollary}
\label{clear} In the hypothesis of Lemma \ref{exact} suppose that
$\sdepth M<\sdepth N$. Then $\sdepth M\geq \sdepth U$.
\end{Corollary}
\begin{proof} If $\sdepth M< \sdepth U$, we get $\sdepth M<\min\{ \sdepth U,\sdepth
N\}$ contradicting Lemma \ref{exact}.
\end{proof}

The analogue of $\ref{depthlemma}(b)$ is wrong.
\begin{Example}
\label{me2} {\em Let $S=K[x,y,z]$, $M=(x,y,z)$. In the exact
sequence $ 0\rightarrow M\rightarrow S\rightarrow K\rightarrow 0$,
we have $\sdepth S=3>\sdepth K=0$ but $\sdepth M=2\neq \sdepth K +
1$.}
\end{Example}
Note that the case treated in Proposition \ref{red}, that is the short exact
sequence $0\rightarrow M\xrightarrow{x_k} M\rightarrow
M/x_kM\rightarrow 0,$ and Lemma \ref{mcii} apparently hints that some analogue of $(b)$
from "Depth Lemma" in the frame of $\sdepth$ might be true.
Unfortunately, this is not the case as shows the following:
\begin{Example}
\label{d2}{\em We have a resolution $0\rightarrow
\Omega^1m\rightarrow S^3\rightarrow m\rightarrow 0$,
 where $S=K[x,y,z]$ and $m=(x,y,z)$. Then $\Omega^1m$ is not free because otherwise $\projdim_S m$ should be $1$,
  which is not true. If $\sdepth \Omega^1m=3$ then follows $\Omega^1m$ free by the elementary Lemma \ref{elem2}.
  Thus $\sdepth \Omega^1m\leq 2=\sdepth m$.}
\end{Example}

However it remains still the problem in general that if for an exact
sequence $0\rightarrow U\rightarrow M\rightarrow N\rightarrow 0$,
$\sdepth M>\sdepth N$ implies $\sdepth U\geq \sdepth N + 1$. In
general this is false (see Example \ref{d2}) but we prove this
result in a special case.
\begin{Lemma}
\label{mcii}If $I\subset S=K[x_1,\ldots,x_n]$ is a monomial complete intersection, then $\sdepth I\\\geq \sdepth S/I + 1.$
\end{Lemma}
\begin{proof}
Let $\{v_1,\ldots,v_m\}$ be the regular sequence of monomials
generating  $I$. Since $\sdepth S/I\\=n-m$, by applying \cite[Theorem
1.1]{As} recursively, and $\sdepth I\geq n-m+1$, by \cite{HSY}, or
\cite[Proposition 3.4]{HVZ}, it follows the desired result.
\end{proof}

In general for any monomial ideal the inequality in above lemma is still an open question. This inequality motivates that $\sdepth I\geq \sdepth J/I+1$ for any two monomial ideals $I\subset J\subset S$. But this inequality does not hold as shows the following example:
\begin{Example}
\label{ref3}Let $S=K[x,y]$, $I=(xy,y^2)$, $J=I+(x^2)$. Then we have $\sdepth J/I=1=\sdepth I=\sdepth J$.
\end{Example}

\begin{Lemma}
\label{elem2}If $M$ is multigraded $S$-module, $S=K[x_1,\ldots,x_n]$ with $\sdepth M=n$ then $M$ is free.
\end{Lemma}
\begin{proof} If $\sdepth\ M=n$, then we have a Stanley
decomposition  of the form $M=\oplus_i u_i S$ and $u_iS$ are free
$S$-modules. The direct sum is of linear spaces but it turns out to
be of free $S$-modules.
\end{proof}

\section{The behavior of sdepth on algebra tensor product}
\begin{Theorem}
 \label{atp}Let $I\subset S_1=K[x_1,\ldots,x_n]$, $J\subset
 S_2=K[y_1,\ldots,y_m]$ be monomial ideals and
 $S=K[x_1,\ldots,x_n,y_1,\ldots,y_m]$. Then $
  \sdepth S_1/I + \sdepth S_2/J \leq \sdepth S/(IS,JS).$
 \end{Theorem}

\begin{proof}

Let
\[
 \mathcal{D}_1:\; S_1/I=\Dirsum_{i=1}^ru_iK[Z_i]
 \]
be a Stanley decomposition of $S_1/I$ such that $\sdepth
\mathcal{D}_1=\sdepth S_1/I$ and
\[
 \mathcal{D}_2:\; S_2/J=\Dirsum_{j=1}^sv_jK[W_j]
 \]
be a Stanley decomposition of $S_2/J$ such that $\sdepth
\mathcal{D}_2=\sdepth S_2/J$. Then we have

\begin{eqnarray*}
S/IS&=&S_1[y_1,\ldots,y_m]/IS\\&=&(S_1/I)
[y_1,\ldots,y_m]\\&=&\Dirsum_{i=1}^ru_iK[Z_i][y_1,\ldots,y_m]\\
&=&\Dirsum_{i=1}^ru_iK[Z_i,y_1,\ldots,y_m]
\end{eqnarray*}
and
\begin{eqnarray*}
S/JS&=&S_2[x_1,\ldots,x_n]/JS\\&=&(S_2/J)[x_1,\ldots,x_n]\\&=&\Dirsum_{j=1}^sv_jK[W_j][x_1,\ldots,x_n]\\
&=&\Dirsum_{j=1}^sv_jK[W_j,x_1,\ldots,x_n].
\end{eqnarray*}
We claim that
$$S/(IS,JS)=\Dirsum_{i,j}u_iv_jK[Z_i,W_j].$$
Let $w\in(IS,JS)^c=S/(IS,JS)$ be a monomial; that is, $w\in S$ and $w\not\in(IS,JS)$. We have $w\not\in IS$ and
$w\not\in JS$. It follows that $w\in(IS)^c$ and $w\in(JS)^c$. Hence there exist $i$ and $j$ such that
 $w\in u_iK[Z_i,y_1,\ldots,y_m]$ and $w\in v_jK[W_j,x_1,\ldots,x_n]$. So we have
 $w\in u_iK[Z_i,y_1,\ldots,y_m]\cap v_jK[W_j,x_1,\ldots,x_n]$ and $u_iK[Z_i,y_1,\ldots,y_m]\cap v_jK[W_j,x_1,\ldots,x_n]=
 u_iv_jK[Z_i,W_j]$, since $u_i\in S_1$ and $v_j\in S_2$.

In order to prove the opposite inclusion, consider a monomial $v\in
u_iv_jK[Z_i,W_j]$. Then $v\in u_iK[Z_i,y_1,\ldots,y_n]\subset
(IS)^c$ and similarly $v\in (JS)^c$. Thus $v\in (IS,JS)^c$. So $S/(IS,JS)=\sum_{i,j}u_iv_jK[Z_i,W_j]$.

    Now we prove that this sum is direct. Let   $i_1,\,i_2\in [r]$ and $j_1,\,j_2\in [s]$ be such
    that $(i_1,j_1)\neq(i_2,j_2)$, let us say $i_1\neq i_2$. Then $u_{i_1}v_{j_1}K[Z_{i_1},W_{j_1}]\cap
u_{i_2}v_{j_2}K[Z_{i_2},W_{j_2}] \subset
u_{i_1}K[Z_{i_1},y_1,\ldots,y_m]\cap
u_{i_2}K[Z_{i_2},y_1,\ldots,y_m]=\{0\}$, which shows our claim. It
follows that
$\sdepth S_1/I+ \sdepth S_2/J\leq \sdepth S/(IS,JS).$
\end{proof}

The following example shows that the inequality in the above theorem can be strict.

\begin{Example}
\label{atpexp}{\em Let $S=K[x_1,x_2,x_3,x_4,x_5,x_6,x_7,x_8]$ be a polynomial ring over the field $K$. Let $I=(x_1x_3,x_1x_4,x_2x_3,x_2x_4)\subset S_1=K[x_1,x_2,x_3,x_4]$ be the ideal of the polynomial ring $S_1$ and $J=(x_5x_7,x_5x_8,x_6x_7,x_6x_8)\subset S_2=K[x_5,x_6,x_7,x_8]$ be the ideal of the polynomial ring $S_2$. Consider the ideal $(IS,JS)\subset S$, then a Stanley decomposition $\mathcal D$ of $S/(IS,JS)$ is

$
\mathcal D:\,\,S/(IS,JS)=K[x_1,x_2,x_5]\oplus x_3K[x_3,x_5,x_6]\oplus x_4K[x_4,x_5,x_6]
\oplus x_6K[x_1,x_2,x_6]\oplus x_7K[x_1,x_2,x_7]\oplus x_8K[x_1,x_2,x_8]\oplus x_3x_4K[x_3,x_4,x_5]\oplus x_3x_7K[x_3,x_7,x_8]\oplus x_3x_8K[x_3,x_4,x_8]\oplus x_4x_7K[x_3,x_4,x_7]\oplus x_4x_8K[x_4,x_7,x_8]\oplus
x_5x_6K[x_1,x_5,x_6]\oplus
x_7x_8K[x_1,x_7,x_8]\\
\oplus x_2x_5x_6K[x_1,x_2,x_5,x_6]\oplus x_3x_4x_6K[x_3,x_4,x_5,x_6]\oplus x_2x_7x_8K[x_1,x_2,x_7,x_8]\\
\oplus x_3x_4x_7x_8K[x_3,x_4,x_7,x_8],
$
hence $\sdepth \mathcal {D}=3$. Note that $\sdepth S_1/I=1$ with a Stanley decomposition $S_1/I=K[x_1,x_2]\oplus x_3K[x_3]\oplus x_4K[x_3,x_4]$. We observe that $\sdepth S_1/I$ can not be greater than one. Similarly we have $\sdepth S_2/J=1$. Hence we obtain that $\sdepth S_1/I + \sdepth S_2/J<\sdepth S/(IS,JS)$.}
\end{Example}
The following corollary is a particular case of \cite[Theorem 1.1]{As}.
\begin{Corollary}
Let $I\subset S$ be a monomial ideal and $u\in S$ is a monomial, which is regular on $S/I$. Then $\sdepth S/(I,u)\geq \sdepth S/I - 1$.
\end{Corollary}
\begin{proof}
Renumbering $x_i\in \supp(u_j)$ for all $u_j\in G(I)$, we may suppose that $I$ is generated by a monomial ideal $J\subset S_1=K[x_1,\ldots,x_r]$ and $u\in S_2=K[x_{r+1},\ldots,x_n]$ for some $1<r<n$. Then $\sdepth S/(I,u)\geq \sdepth S_1/J + \sdepth S_2/(u)$, by Theorem \ref{atp}. Since $\sdepth S_2/(u)=n-r-1$ and $\sdepth S/I=\sdepth S_1/J+n-r$, by \cite[Lemma 1.2]{As}, it follows that $\sdepth S/(I,u)\geq \sdepth S/I - 1$.
\end{proof}

 In the analogue of Theorem \ref{atp} for depth we have equality, that is \\$\depth S/(IS,JS)=\depth S_1/I+\depth S_2/J$ (see \cite[Theorem 2.2.21]{V}). We note that if Stanley's conjecture hold for the modules $S_1/I$ and $S_2/J$ it holds also for $S/(IS,JS)$.

\end{document}